\documentclass[11pt]{article}

\usepackage{amsfonts,amsmath,amsthm,latexsym,color,epsfig,mathrsfs,enumerate}
\newtheorem{thm}{Theorem}
\newtheorem{dfn}[thm]{Definition}
\newtheorem{lem}[thm]{Lemma}
\newtheorem{cor}[thm]{Corollary}

\newtheorem{conjecture}[thm]{Conjecture}
\newtheorem{rem}[thm]{Remark}
\newtheorem{obs}[thm]{Observation}
\newtheorem{clm}[thm]{Claim}

\DeclareMathOperator{\reg}{reg}
\DeclareMathOperator{\per}{per}

\title{\vspace{-1cm} Counting Hamilton decompositions of oriented graphs}
\author{Asaf Ferber\thanks{Department of Mathematics, MIT, USA. Email:
ferbera@mit.edu.} \and Eoin Long\thanks{School of Mathematical Sciences, Tel Aviv University, Tel Aviv, Israel. Email: eoinlong@post.tau.ac.il.} \and Benny Sudakov\thanks{Department of Mathematics, ETH, Z\"urich, Switzerland. Email: benjamin.sudakov@math.ethz.ch.}}
\date{}

\begin{document}

\maketitle

\begin{abstract}
A Hamilton cycle in a directed graph $G$ is a cycle that passes through every vertex of $G$.
A Hamiltonian decomposition of $G$ is a partition of its edge set into disjoint Hamilton cycles.
In the late $60$s Kelly conjectured that every regular tournament has a Hamilton decomposition. This conjecture was recently settled  by K\"uhn and Osthus \cite{KO}, who proved more generally that 
every $r$-regular $n$-vertex oriented graph $G$ (without antiparallel edges) with $r=cn$ for some fixed $c>3/8$ has a Hamiltonian decomposition, provided $n=n(c)$ is sufficiently
large. In this paper we address the natural question of estimating the number of such decompositions of $G$ and show that this number is $n^{(1-o(1))cn^2}$. 
In addition, we also obtain a new and much simpler proof for the approximate version of Kelly's conjecture.
\end{abstract}

\section{Introduction}

A \emph{Hamilton cycle} in a graph or a directed graph $G$ is a cycle
passing through every vertex of $G$ exactly once, and a graph
is \emph{Hamiltonian} if it contains a Hamilton cycle. Hamiltonicity
is one of the most central notions in graph theory, and has been
intensively studied by numerous researchers in recent
decades. The decision problem of whether a given graph contains a Hamilton cycle is known to be $\mathcal{NP}$-hard and in fact, already appears on Karp's original list
of 21 $\mathcal{NP}$-hard problems \cite{karp1972reducibility}.
Therefore, it is important to find general sufficient conditions for
Hamiltonicity (for a detailed discussion of this topic we refer the interested reader to two surveys of K\"uhn and Osthus \cite{KOI,KOSurvey}).

In this paper we discuss Hamiltonicity problems for directed graphs. A \emph{tournament} $T_n$ on $n$ vertices is an orientation of an $n$-vertex complete graph $K_n$. The tournament is \emph{regular} if all in/outdegrees are the same and equal $(n-1)/2$. It is an easy exercise to show that every tournament contains a Hamilton path (that is, a directed path passing through all the vertices). Moreover, one can further show that a regular tournament contains a Hamilton cycle. 

A tournament is a special case of a more general family of directed graphs,  so called \emph{oriented} graphs. An oriented graph is a directed graph obtained by orienting the edges of a simple graph (that is, a graph without loops or multiple edges). Given an oriented graph $G$, let $\delta^+(G)$ be its minimum outdegree, $\delta^-(G)$ be its minimum indegree and let
the \emph{semi-degree} $\delta^0(G)$ be the minimum of $\delta^+(G)$ and $\delta^-(G)$. A natural question, originally raised by Thomassen in the late 70s, asks to determine the minimum semi-degree which ensures Hamiltonicity in the oriented setting. Following a long line of research,  Keevash, K\"uhn and  Osthus \cite{KeeKO} settled this problem, showing that $\delta^0(G)\geq \lceil \frac{3n-4}{8}\rceil$ is enough to obtain a Hamilton cycle in any $n$-vertex oriented graph. A construction showing that this is tight was obtained much earlier by H\"aggkvist \cite{Haa}.

Once Hamiltonicity of $G$ has been established, it is natural to further ask whether $G$ contains many edge-disjoint Hamilton cycles or even a \emph{Hamilton decomposition}. A Hamilton decomposition is a collection of edge-disjoint Hamilton cycles covering all the edges of a graph. In the late $60$s, Kelly conjectured (see \cite{KOI,KOSurvey} and their references) that every regular tournament has Hamilton decomposition. Kelly's Conjecture has been studied extensively in recent decades, and quite recently was settled for large tournaments in a remarkable tour de force by K\"uhn and Osthus \cite{KO}. In fact, K\"uhn and Osthus \cite{KO} proved the following  stronger statement for dense \emph{$r$-regular} oriented graphs (that is, oriented graphs with all in/outdegrees equal to $r$).

\begin{thm}
\label{Kelly}
  Let $\epsilon>0$ and let $n$ be a sufficiently large integer. Then, every $r$-regular oriented graph $G$ on $n$ vertices with $r\geq 3n/8+\epsilon n$ has a Hamilton decomposition.
\end{thm}

\noindent
The bound on $r$ in this theorem is best possible up to the additive term of $\epsilon n$. Indeed, as we already mentioned above, if $r$  is smaller than $3n/8$ then $G$ may not even be Hamiltonian.

Counting various combinatorial objects has a long history in Discrete Mathematics and such problems have been extensively studied.  
Motivated by Theorem \ref{Kelly}, in this paper we consider the number of distinct Hamilton decompositions of dense regular oriented graphs. One can obtain an upper bound for this question by using the famous Minc conjecture, established by Br\'egman \cite{Bregman}, which provides an upper-bound on the permanent of a matrix $A$. Let $S_n$ be the set of all permutations of the set $[n]$. The permanent of an $n\times n$ matrix $A$ is defined as $\per(A)=\sum_{\sigma\in S_n}\prod_{i=1}^nA_{i\sigma(i)}$. Note that every permutation $\sigma\in S_n$ has a cycle representation which is unique up to the order of cycles. When $A$ is a $0-1$ adjacency matrix of an oriented graph (that is $A_{ij}=1$ iff $\vec{ij}\in E(G)$), every non-zero summand in the permanent is $1$ and it corresponds to a collection of disjoint cycles covering all the vertices. Hence, the permanent counts the number of such cycle factors and, in particular,  gives un upper bound on the number of Hamilton cycles in the corresponding graph. For an $r$-regular oriented graph $G$ with adjacency  matrix $A$, Br\'egman's Theorem asserts that
$$\per(A)\leq (r!)^{n/r}=(1-o(1))^n(r/e)^n.$$

Therefore, $G$ has at most $(1-o(1))^n(r/e)^n$ Hamilton cycles. Note that upon removing the edges of such a cycle from $G$, we are left with an $(r-1)$-regular oriented graph $G'$. Again by Br\'egman's Theorem, $G'$ contains at most $(1-o(1))^n((r-1)/e)^n$ distinct Hamilton cycles. Repeating this process and taking the product of all these estimates, we deduce that $G$ has at most $$\left((1+o(1))\frac{r}{e^2}\right)^{rn}$$ Hamilton decompositions. When $r$ is linear in $n$ this behaves asymptotically as $n^{(1-o(1))rn}$.

Our  first result  gives a corresponding  lower  bound,  which  together  with  the  above estimates determine asymptotically the number of Hamiltonian decompositions of dense regular oriented graphs. It is worth drawing attention to the fact that our result shows that all such graphs have roughly the same number of Hamilton decompositions.
\begin{thm}
\label{main2}
  Let $c>3/8$ be a fixed constant, let $\epsilon>0$ be an arbitrary small constant, and let $n$ be a sufficiently large integer. Then, every $cn$-regular oriented graph $G$ on $n$ vertices contains at least
  $n^{(1-\epsilon)cn^2}$
  distinct Hamilton decompositions.
\end{thm}

\noindent
The main step in the proof of this theorem is to construct many almost Hamilton decompositions, each of which can be further completed to a full decomposition. This is done by extending some ideas from \cite{FKL} and differs from the approach used in \cite{KO}. In particular, we obtain a new and much simpler proof for the approximate version of Kelly's conjecture, originally established by K\"uhn, Osthus and Treglown in \cite{KOT}. Furthermore, note that a Hamilton decomposition of a regular tournament also gives a Hamilton decomposition of the underlying complete (undirected) graph.
Therefore Theorem \ref{main2} implies that, for odd $n$, the $n$-vertex complete graph has $n^{(1-o(1))n^2/2}$ Hamilton decompositions. This estimate, together with more general results concerning counting Hamilton decompositions of various dense regular graphs, was recently obtained in \cite{GLS}.

Another natural problem studied in this paper concerns how many edge-disjoint Hamilton cycles one can find in a given (not necessarily regular) oriented graph. Observe that if an oriented graph $G$ contains $r$ edge-disjoint Hamilton cycles, then their union gives a spanning, $r$-regular subgraph of $G$. We refer to such a subgraph as an $r$-\emph{factor} of $G$. Given an oriented graph $G$, let $\reg(G)$ be the maximal integer $r$ for which $G$ contains an $r$-factor.
Clearly, $G$  contain at most $\reg(G)$ edge-disjoint Hamilton cycles. We propose the following conjecture which, if true, is best possible. 

\begin{conjecture}\label{conj}
  Let $c>3/8$ be a fixed constant and let $n$ be sufficiently large. Let $G$ be an oriented graph on $n$ vertices with $\delta^0(G)\geq cn$. Then, $G$ contains $\reg(G)$ edge-disjoint Hamilton cycles.
\end{conjecture}

Our second result gives supporting evidence for this conjecture, proving that such oriented graphs $G$ contain $(1-o(1))\reg(G)$ edge-disjoint Hamilton cycles.

\begin{thm}
  \label{main}
Let $c>3/8$ and $\varepsilon>0$ be fixed constants and let $n$ be sufficiently large. Let $G$ be an oriented graph on $n$ vertices with $\delta^0(G)\geq cn$. Then, $G$ contains a collection of $(1-\varepsilon)\reg(G)$ edge-disjoint Hamilton cycles.
\end{thm}

\noindent
This theorem follows immediately from our proof of Theorem \ref{main2}. For a regular tournament Theorem \ref{main} implies an approximate version of Kelly's Conjecture from \cite{KOT}. \vspace{1.5mm}

{\bf Notation:}
Given an oriented graph $G$ and a vertex $v\in V(G)$, we let $d_G^+(v)$ and $d_G^-(v)$ to denote the out- and in-degree of $v$, respectively. We omit the subscript $G$ whenever there is no chance of confusion. We also define $\delta^+(G):=\min_v d^+(v)$, $\delta^-:=\min_vd^-(v)$, $\Delta^+(G):=\max_vd^+(v)$, $\Delta^-(G):=\max_v d^-(v)$, and set $\delta^0(G)=\min\{\delta^+(G),\delta^-(G)\}$ and $\Delta^0(G)=\max\{\Delta^+(G),\Delta^-(G)\}$. We also write $a\pm b$ to denote a value which lies in the interval $[a-b,a+b]$.

\section{Tools}

In this section we have collected a number of tools to be used in the proofs of our results.

\subsection{Chernoff's inequality}
	
	Throughout the paper we will make extensive use of the following well-known bound on the
	upper and lower tails of the Binomial distribution, due to Chernoff
	(see Appendix A in \cite{AlonSpencer}).

		\begin{lem}[Chernoff’s inequality]
		Let $X \sim Bin(n, p)$ and let
		${\mathbb E}(X) = \mu$. Then
			\begin{itemize}
				\item
				${\mathbb P}[X < (1 - a)\mu ] < e^{-a^2\mu /2}$
				for every $a > 0$;
				\item ${\mathbb P}[X > (1 + a)\mu ] <
				e^{-a^2\mu /3}$ for every $0 < a < 3/2$.
			\end{itemize}
		\end{lem}
	
\noindent
	\begin{rem}\label{rem:hyper} These bounds also hold when $X$ is
	hypergeometrically distributed with
	mean $\mu $.
\end{rem}

\subsection{Perfect matchings in a bipartite graph}

Here we present a number of results related to  perfect matchings in bipartite graphs.
The first result is a criterion for the existence of $r$-factors in bipartite graphs, due to Gale and Ryser (see \cite{Gal}, \cite{Ry}).

\begin{thm}
  \label{gale}
Let $G=(A\cup B,E)$ be a bipartite graph with $|A|=|B|=m$, and let $r$ be an integer. Then $G$ contains an $r$-factor if and only if for all $X\subseteq A$ and $Y\subseteq B$
\begin{align*}
  e_G(X,Y)\geq r(|X|+|Y|-m).
\end{align*}
\end{thm}

Next we present Br\'egman's Theorem which provides an upper bound for the number of perfect matchings in a bipartite graph based on its degrees (see e.g. \cite{AlonSpencer} page 24).

\begin{thm}(Br\'egman's Theorem)
\label{Bregman}
Let $G=(A\cup B,E)$ be a bipartite graph with $|A| = |B|$. Then the number of perfect matchings in $G$ is at most
$$\prod_{a\in A} (d_G(a)!)^{1/d_G(a)}.$$
\end{thm}

\begin{rem}
  \label{Bregman wrt max degree}
  It will be useful for us to give an upper bound with respect to the maximum degree of $G$. Suppose that $|A|=|B|=m$ and let $\Delta:=\Delta(G)$. Using Theorem \ref{Bregman} and Stirling's approximation, one obtains that the number of perfect matchings in $G$ is at most
$$(\Delta!)^{m/\Delta}\leq (8\Delta )^{m/\Delta }\left(\frac{\Delta}{e}\right)^m.$$
\end{rem}

Lastly, we require the following result which provides a lower bound for the number of perfect matchings in a regular bipartite graph. This result is known as the Van Der Waerden Conjecture, and it was proven by Egorychev \cite{Egorychev}, and independently by Falikman \cite{Falikman}.

\begin{thm} (Van Der Waerden's Conjecture)
\label{VDW}
Let $G=(A\cup B,E)$ be a $d$-regular bipartite graph with both parts of size $m$. Then the number of perfect matchings in $G$ is at least
$$d^m\frac{m!}{m^m}\geq \left(\frac{d}{e}\right)^m.$$
\end{thm}

 \subsection{Hamilton paths, cycles and absorbers}

	We make use of the following theorem of Keevash, K\"uhn and Osthus \cite{KeeKO}.

	\begin{thm}
  		\label{thm:exact semidegree condition}
	Every $n$-vertex oriented graph $G$ with $\delta ^0(G) \geq (3n-4)/8$ contains
	a Hamilton cycle, provided $n$ is sufficiently large.
	\end{thm}	
		
	We also make use of the following related result of Kelly, K\"uhn and Osthus,
	which follows immediately from the proof of the main theorem in \cite{KKO}.

	\begin{thm}
  		\label{thm:Hamilton connectivity}
	Let $c>3/8$ be a constant and $n$ be sufficiently large. Suppose that $G$ is an 	
	oriented graph on $n$ vertices with $\delta^0(G)\geq cn$, and let $x,y\in V(G)$ be
	any two distinct vertices. Then there is a Hamilton path in $G$ with $x$ as its
	 starting point and $y$ as its final point.
	\end{thm}

    Before describing the next tool we need the following definition.

    \begin{dfn}
	Given an $n$-vertex oriented graph $G$, a subgraph $D\subseteq G$
	is said to be a \emph{$\delta $-absorber} if, for any given $d$-regular
	subgraph $T$ which is edge-disjoint from $D$ with
	$d \leq \delta n$, the oriented graph $D \cup T$ has
	a Hamilton decomposition.
	\end{dfn}

	The following result is the main ingredient in the seminal paper of
	K\"uhn and Osthus in which they solved Kelly's conjecture \cite{KO}. Roughly speaking,
	the theorem states that there are $\delta$-absorbers for arbitrarily small $\delta$
	in any sufficiently large regular oriented graph. This refined version follows immediately from the directed version of Theorem 3.16 in \cite{KOI}. 
	\begin{thm}
		\label{thm: Kuhn-Osthus Absorbing}
		Let $\varepsilon>0$ and $c>3/8$ be two constants. Then, there is $\delta >0$
		such that for sufficiently large $n$ the following holds. Suppose that $G$ is an $n$-vertex oriented graph with $\delta^{0}(G)\geq cn$. Then $G$ contains a
		$\delta $-absorber 	$A$ as an oriented
		subgraph, where $A$ is $r$-regular with $r\leq \varepsilon n$.
	\end{thm}

\section{Almost Hamilton decompositions of special oriented graphs}

Our aim in this section is to show how certain special oriented graphs can be
 almost decomposed into Hamilton cycles.

\subsection{Completing one Hamilton cycle}

The following simple lemma will allow us to complete
disjoint directed paths into Hamilton cycles.

\begin{lem}
	\label{lem: completion of path covers to Ham cycles}
	Let $c>3/8$ and $a, N \in {\mathbb N}$ with $a \ll \frac{N}{\log N}$ and $N$ sufficiently large. Let $F$ be an oriented graph with
	$|V(F)| = N$ and $\delta^0(F)\geq cN$. Let $\{P_i\}_{i\in [a]}$ be a collection of vertex disjoint oriented paths, each of which is disjoint to $F$. Let $x_i$ and $y_i$ to denote the first and last vertices of $P_i$, for each $i$, and assume that $d^-(x_i,F),d^+(y_i,F)\geq 2a$.
Then there is a cycle $C$ with the following properties:
\begin{enumerate}
  \item Each $P_i$ appears as a segment of $C$;
  \item $F \subseteq V(C)$.
\end{enumerate}
\end{lem}

\begin{proof}[Proof of Lemma \ref{lem: completion of path covers to Ham cycles}]
	For each $i\in [a]$ select
	$t_i \in N^-(x_i)$ and $s_i \in N^+(y_i)$ such that all $2a$ vertices are distinct. Note that this
	is possible as $d^-(x_i,F), d^+(y_i,F) \geq 2a$. Let
	$S = \{s_i: i\in [a]\}$, $T = \{t_i: i\in [a]\}$ and
	$W = V(F)$.
	
	Let us create a partition of $W$ into
	$a$ sets, $W_1,\ldots, W_a$, by assigning $s_i$
	and $t_{i+1}$ to $W_i$ for all $i\in [a]$ (taking $a+1$ to be
	$1$) and by randomly assigning each vertex $v\in W \setminus (S \cup T)$
	to one of the sets uniformly and independently
	at random. Now, let $\varepsilon_0 = (c-3/8)/4 > 0$ and consider the events:
		\begin{align*}
			A
				& =
			``|W_i| \in (1 \pm \varepsilon _0)
					\frac {|W|}{a} \mbox{ for all }
					i\in [a]"\\
			B
				& =
			``d^{\pm }_{F[W_i]}(v) \geq
			\big (c - \varepsilon _0 \big )
			\frac{|W|}{a} \mbox { for all }
			v\in W \mbox { and } i\in [a]".
		\end{align*}
	As
	${\mathbb E}(|W_r|) = \frac {|W|}{a}$, using that $N\gg a\log a$ and
	Chernoff's inequality, we obtain
		\begin{equation}
			\label{equation: control of U_k sets}
			{\mathbb P}[{A}^c]
			\leq
				2m\exp \Big (-\frac {\varepsilon _0^2|W| }{3a} \Big )
			= o(1).
		\end{equation}
Also, as $\delta^0(F)\geq cN=c|W|$ and all but at most $2a$ vertices were assigned randomly, we have
$${\mathbb E}(d^{\pm }(v,W_i))	\geq \frac{c|W|-2a}{a}=c\frac{|W|}{a}-2.$$
Again using that $|W|\gg a \log N$ together with Chernoff's inequality, we have
		\begin{equation}
			\label{equation: contol of degrees of vertices in S in U_k}
			{\mathbb P} [B ^c]
				\leq
			2N \exp \Big (- \Theta \big (\frac {\varepsilon _0^2 |W|}{a} \big ) \Big ) = o(1).
		\end{equation}
	Combining \eqref{equation: control of U_k sets} with
	\eqref{equation: contol of degrees of vertices in S in U_k}
	we conclude ${\mathbb P}( A\cap B ) >0$.
	Fix a partition $W_1,\ldots ,W_{a}$
	such that $A\cap B$ holds.\vspace{2mm}
	
	To complete the proof, set $F_{i} := F[W_i]$ for each $i\in [a]$. As $A \cap B$ holds, we have
		\begin{equation*}
			\delta ^{0}(F_i)
				\geq
			\big (c - \varepsilon _0 \big ) \frac {|W|}{a}
				\geq
			\big (c - 3\varepsilon _0 \big ) |V(F_i)| = (3/8 + \varepsilon _0) |V(F_i)|.
		\end{equation*}
	Therefore, using that $|V(F_i)| \geq (1-\varepsilon _0)|W|/a \geq
	N/2a \gg \log N$ and $N$ is sufficiently large, it follows from Theorem \ref{thm:Hamilton connectivity} that $F_{i}$ contains
	a Hamilton path ${I}_i$ from $s_{i}$
	to $t_{i+1}$, for each $i$. All in all, the cycle $C=P_1I_1P_2I_2\ldots P_aI_aP_1$ (with the connecting edges $y_is_i$ and $t_{i+1}x_{i+1}$) gives the desired cycle. This completes the proof of the lemma.
\end{proof}

\subsection{Completing `many' edge-disjoint Hamilton cycles}

Next we will show how to repeatedly apply Lemma \ref{lem: completion of path covers to Ham cycles} to obtain `many' edge-disjoint Hamilton cycles. Before stating this result we introduce the following definitions.

\begin{dfn} Let $G$ be an oriented graph.
\begin{enumerate}
	\item A \emph{path cover} of $G$ of size $a$ is a collection
	of $a$ vertex disjoint directed paths in $G$ which cover all vertices in $V(G)$.
\item An $(a,t)_{\mathcal P}$-\emph{family} is a collection of $t$ edge-disjoint paths covers of $G$, each of which is of size at most $a$.
	\item Let ${\mathcal P}(G, a, t)$ denote the set of all $(a,t)_{\mathcal P}$-families in $G$.
\item Given ${\bf P} \in {\mathcal {P}}(G, a, t)$, let $G_{\bf P}$ 	denote the
	oriented subgraph $G_{\bf P} = \bigcup _{P\in \bf P} E(P)$.
\end{enumerate}
\end{dfn}

\noindent \emph{Remark:} The above definitions  include the possibility of
paths of length $0$, i.e. isolated vertices.\vspace{2mm}

One can think about a path cover ${P}$ of small size as an
`almost Hamilton cycle', in the sense that by adjoining a small number of edges to ${P}$ we can obtain a Hamilton cycle. Our aim in the following lemma is to show how, given `many' edge-disjoint path covers, one can build `many' edge disjoint Hamilton cycles.

\begin{lem}
	\label{lem: completion of many hamilton cycles}
	Let $c>3/8$ and let $a, b, n, s, t \in {\mathbb N}$ with
	$t+a \log n \ll s \ll n$.
	Suppose that $H$ is an $n$-vertex oriented
	graph with partition $V(H) = U \cup W$,
	where $|W| = s$, with the following properties:
		\begin{enumerate}
			\item There is ${\bf P}\in \mathcal P(H[U],a,t)$;		
			\item $\delta ^0(H_{\bf P}[U])
			\geq t - b$;
			\item \label{lem: completion iii}
			$d^{\pm }(u,W)>
			2a + b$ for all $u\in U$;
			\item The oriented subgraph $F = H[W]$ satisfies
			$\delta^0(F) \geq c
			|W|$;
		\end{enumerate}
	Then $H$ contains a family ${\cal C} = \{C_1,\ldots ,C_t\}$ of $t$ edge
	disjoint Hamilton cycles, where each cycle $C_i$ contains all the paths in $\mathcal P_i$ as segments.
\end{lem}

\begin{proof}
	Let ${\bf P}\in \mathcal P(H[U],a,t)$ and write $${\bf P}=\{{\cal P}_j \mid j\in[t]\}.$$
	For each $j$, let $\mathcal P_j = \{P_{j,r}\}_{r \in [R_j]}$  denote the collection
	of all directed paths in the path cover $\mathcal P_j$.	As $\mathcal P_j$ has size
	at most $a$ we have $R_j \leq a$.
	
	Now we wish to turn each $\mathcal P_j$ into a Hamilton cycle $C_j$ of $H$ in such a
	way that
		\begin{enumerate}[(i)]
			\item all the paths in $\mathcal P_j$ are segments of $C_j$, and
			\item $C_i$ and $C_j$ are edge-disjoint for all $i\neq j$.
		\end{enumerate}
This will be carried out over a sequence of steps where in step $j$ we have already selected $C_1,\ldots,C_{j-1}$, and the cycle $C_j$ is chosen by showing that the oriented graph $H_j=H\setminus \big (\bigcup_{i\leq j-1} E(C_i)\big )$ satisfies the requirements of Lemma \ref{lem: completion of path covers to Ham cycles}. Let us fix $c>c'>3/8$.

Suppose that we have already found $C_1,\ldots ,C_{j-1}$ and we wish to find $C_j$. Let $x_i$ and $y_i$ denote the start and end vertices of $P_{j,i}$, for all $i\leq R_j$. First note that by property $3$, each vertex $u\in \{x_i,y_i \mid i\leq R_j\}$ satisfies $d^{\pm}(u,W)>2a + b$. By property $2$, each vertex $v$ appears the first vertex of at most $b$ paths and as the last vertex of at most $b$ paths (otherwise $v$ would have in-degree or out-degree less than $t-b$ in $H_{\bf P}(U)$). Therefore,
for all $u\in U$ we have
$$d^{\pm}_{H_j}(u,W)\geq 2a.$$
Second, as the edges of less than $j$ Hamilton cycles have been deleted from $H$, from property $4.$ we find that $F_j=H_j[W]$ satisfies $\delta^0(F_j)\geq c|W| - j +1 \geq c'|W|$, using $|W| = s \gg t \geq j$. Lastly, we have $|W| = s \gg a \log n \gg a\log s$ by hypothesis.

All combined, we have shown that the graph $H_j$ satisfies the conditions of Lemma \ref{lem: completion of path covers to Ham cycles} with $N = |W|$. Therefore Lemma \ref{lem: completion of path covers to Ham cycles} guarantees the cycle $C_j$ exists. Thus we can find $C_1,\ldots, C_t$, as required.
	\end{proof}

\section{Path covers of oriented graphs}

In the previous section we have shown how to extend edge disjoint path covers to edge disjoint Hamilton cycles in certain special oriented graphs. In this section we will show how to located such path covers, using a number of well-known matching results. The main result of the section is the following:

\begin{lem}
	\label{lem: path cover lemma}
Let $m,r \in {\mathbb N}$ with $ r \geq m^{49/50}$ and $m$ sufficiently large. Suppose that $H$ is an $m$-vertex oriented graph with
$$r-r^{3/5}\leq \delta^0(H)\leq \Delta^0(H)\leq r+r^{3/5}.$$
	Then, taking $a = m/\log ^4m$ and $t =r - m^{24/25}\log m$, the following hold:
		\begin{enumerate}
			\item There is a set ${\bf S} \subseteq {\mathcal P}(H,a,t)$ with
			 $|{\bf S}| \geq r^{(1-o(1))rm}$;
			\item For all ${\bf P} \in {\bf S}$ the oriented subgraph $H_{\bf P}$
			satisfies $\delta^0(H_{\bf P})\geq r - m/{\log ^{4}m}$. 		
		\end{enumerate}
\end{lem}

\subsection{Finding $r$-factors in bipartite graphs}

We show that given a dense bipartite graph $G=(A\cup B,E)$ which is `almost regular', $G$ contains a spanning $r$-regular subgraph (an \emph{$r$-factor}), with $r$ very close to $\delta(G)$.

\begin{lem}
  \label{factor}
  Let $\alpha\geq 1/2$, $m \in {\mathbb N}$ and $\xi=\xi(m)\geq 0$. Suppose $G=(A\cup B,E)$ is a bipartite graph with $|A|=|B|=m$ and $\alpha m+\xi\leq \delta(G)\leq \Delta(G)\leq \alpha m+\xi +{\xi^2}/{m}$. Then $G$ contains an $\alpha m$-factor.
\end{lem}

\begin{proof}
By Theorem \ref{gale}, to prove the lemma it suffices to show that for all $X \subset A$ and $Y \subset B$ we have
	\begin{equation}
		\label{eq:gale}
		e_G(X,Y) \geq r(|X| + |Y| - m).	
	\end{equation}
Given such sets $X$ and $Y$, let $x=|X|$ and $y=|Y|$. We may assume that $x\leq y$, as the case $y\leq x$ follows by symmetry. We will make use of the following two trivial estimates for $e_G(X,Y)$:
\begin{enumerate}[(i)]
	\item \label{small x} $e_G(X,Y)\geq x(\delta(G)+y-m)$;
	\item \label{large x}
	$e_G(X,Y)=e_G(X,B)-e_G(X,B\setminus Y) \geq \delta(G)x-\Delta(G)(m-y)$.
\end{enumerate}

The required bound follows from the following cases.\vspace{1.5mm}

{\bf Case 1:} $x+y\leq m$. In this case \eqref{eq:gale} trivially holds.\vspace{1.5mm}

{\bf Case 2:} $x\leq y$ and $x\leq \delta(G)$. In this case, note that since $y-m\leq 0$ we obtain
$$x(\delta(G)+y-m)\geq \delta(G)(x+y-m).$$
which by \eqref{small x} proves \eqref{eq:gale}.\vspace{1.5mm}

{\bf Case 3:} $x\leq y$, and $x>\delta(G)$. Observe that in this case since $\alpha\geq 1/2$ we have
\begin{align}
  \label{x+y}
  x+y-m\geq 2\delta(G)-m\geq 2\xi.
\end{align}
Also, from \eqref{large x}, we have
\begin{align}
\label{fin}
  e_G(X,Y)\geq \delta(G)x-\Delta(G)(m-y)
  \geq \alpha m(x+y-m)+\xi(x+y-m)-\frac{\xi^2}{m}(m-y).
\end{align}
Combining \eqref{x+y} with \eqref{fin} and using that $x+y \geq m$, we conclude that
$$\alpha m(x+y-m)+\xi(x+y-m)-\frac{\xi^2}{m}(m-y)\geq \alpha m(x+y-m)+2\xi^2-\xi^2\geq \alpha m(x+y-m),$$
which again proves \eqref{eq:gale}. This completes the proof.
\end{proof}

Using the previous lemma we obtain the following corollary, which shows that by adjoining a small number of edges to an almost regular bipartite graph,  one can obtain a regular bipartite graph.

\begin{cor}
  \label{lemma:Making sparse regular}
  Let $d,m \in {\mathbb N}$, $d \leq m/2$ and $\xi = \xi (m)\geq 0$. Suppose that $G=(A\cup B,E)$ is a
  bipartite graph with $|A|=|B|=m$ and that
  $d-\xi-\xi^2/m \leq \delta(G)\leq \Delta(G)\leq d-\xi$. Then there is a
  bipartite $d$-regular graph $H=(A\cup B,E')$ which contains $G$ as a subgraph.
\end{cor}

\begin{proof}
  Given $G$ as in the lemma, consider the graph $G^c=(A\cup B,E^*)$ where  $e\in E^*$ if and only if $e\notin E$. Clearly $m-d+\xi\leq \delta(G^c)\leq \Delta(G^c)\leq m-d+\xi+\xi^2/m$. Therefore, Lemma \ref{factor} guarantees a $(m-d)$-regular subgraph $S\subseteq G^c$. Letting $H:=S^c$ completes the proof.
\end{proof}

\subsection{Small subgraphs contribute many edges to few matchings}

\begin{lem}
  \label{many large matchings}
Let $m,r \in {\mathbb N}$ with $r \geq m^{24/25}$ and $m$ sufficiently large. Suppose that $G=(A\cup B,E)$ is a bipartite graph with $|A|=|B|=m$ and that $E=E_1\cup E_2$ is a partition of $E$. For $i\in \{1,2\}$ let $H_i$ be the spanning subgraph of $G$ induced by the edges in $E_i$. Suppose also that:
\begin{enumerate}
  \item $G$ is $r$-regular, and
  \item $d_{H_2}(v)\leq 2m^{5/6}$ for all $v\in A\cup B$.
\end{enumerate}
Then $G$ contains at least $(1-o(1))\left(\frac{r}{e}\right)^{m}$ perfect matchings, each with at most $m^{7/8}$ edges from $E_2$.
\end{lem}

\begin{proof}
Set $s = 2m^{5/6}$ and $\ell = m^{7/8}$. First note that since $G$ is $r$-regular, by Theorem \ref{VDW}, the number of perfect matchings in $G$ is at least $\left(\frac{r}{e}\right)^m.$ Therefore it is enough to show
that at most $o(1) (r/e)^m$ matchings of $G$ contain at least $\ell $ edges from $E_2$.

Now given a matching $M\subseteq E_2$ of size $\ell $, let $G'$ be the subgraph of $G$ obtained by deleting the vertices covered by $M$. Clearly $\Delta(G')\leq r$ and $|V(G')|=2(m-\ell)$. By Remark \ref{Bregman wrt max degree} it follows that the number of ways to complete $M$ into a perfect matching is at most
  $$(8\Delta)^{\frac{m-\ell}{r}}\left(\frac{r}{e}\right)^{m-\ell}
  	\leq
  (8r)^{m^{1/25}}\Big ( \frac{e}{r} \Big )^{\ell } \left(\frac{r}{e}\right)^{m}.$$

However, the number of matchings of size $\ell$ in $H_2$ is at most
$\binom{m}{\ell}s^{\ell} \leq (ems/\ell)^{\ell}$. Therefore the number of perfect matchings of $G$ with at least $\ell$ edges from $E_2$ is at most
\begin{align*}
	(8r)^{m^{1/25}} \Big ( \frac{e^2ms}{r\ell} \Big )^{\ell } \left(\frac{r}{e}\right)^{m}
		\leq
	(8m)^{m^{1/25}} \Big ( \frac{2e^2m^{1/25}}{m^{1/24}} \Big )^{m^{7/8}} \left(\frac{r}{e}\right)^{m}
		=
	o(1) \left(\frac{r}{e}\right)^{m}.
\end{align*}
This completes the proof of the lemma.
\end{proof}

\subsection{Decomposing almost regular bipartite graphs into large matchings}

The following definition is convenient.

\begin{dfn} Let $G=(A\cup B,E)$ be a bipartite graph.
\begin{enumerate}
\item Given two integers $a$ and $t$, we define an $(a,t)_{\mathcal M}$-\emph{family} in $G$ to be a collection of $t$ edge-disjoint matchings in $G$, each of which of size at least $a$.
\item Let $\mathcal {M}(G, a, t)$ denote the collection of all $(a,t)_{\mathcal M}$-families in $G$.
\item  Given ${\bf M} \in {\mathcal {M}}(G, a, t)$, we let $G_{\bf M}$ to denote the spanning subgraph of $G$ consisting of the edge set $\bigcup _{M\in {\bf M}} E(M)$.
\end{enumerate}
\end{dfn}

Our main aim in the following lemma is to show that if $G = (A\cup B,E)$ is an almost
$r$-regular bipartite graph with $|A|=|B|$, then for many elements ${\bf M}\in {\mathcal {M}}(G,a,t)$, where $a \approx |A|$ and $t \approx r$, the graph $G_{\bf M}$ is also almost regular.

\begin{lem}
	\label{lem: partial matching lemma}
	Let $\varepsilon > 0$ and $m, r \in {\mathbb N}$ with $m$ sufficiently large and
	$2m^{24/25} \leq r \leq
	(1-\varepsilon )m/2$. Suppose that $G=(A\cup B,E)$ is a bipartite graph with $|A|=|B|=m$
	and $r\leq \delta(G)\leq \Delta(G)\leq r+r^{2/3}$. Then, taking $t=r-m^{24/25}$ and $a=m-m^{7/8}$, the following hold:
		\begin{enumerate}
			\item There is $\mathcal M \subset {\mathcal M}(G,a,t)$,
			with $|\mathcal M|=r^{(1-o(1))rm};$
			\item For each ${\bf M} \in \mathcal M$, the subgraph
			$G_{\bf M} $ has minimum degree at least $t - 2m^{5/6}$.
		\end{enumerate}
\end{lem}

\begin{proof} Set $\xi = m^{5/6}$ and $r' = r + \xi + \xi ^2 /m$. Then, using that $r^{2/3} \leq m^{2/3} = \xi ^2 / m$, combined with the hypothesis of the lemma, we have
	\begin{equation*}
			r' - \xi - \xi ^2 /m
				=
			r	
				\leq
			\delta(G)
				\leq
			\Delta(G)
				\leq
			r+r^{2/3}
				=
			r' - \xi.
	\end{equation*}
Thus by Corollary \ref{lemma:Making sparse regular} there is an $r'$-regular graph $H = (A\cup B, E')$ which contains $G$ as a subgraph.

Set $E_1:=E(G)$ and $E_2:=E(H)\setminus E_1$. By the above, we have
\begin{align}
\label{degrees in E2} d_{E_2}(v)\leq r'-r=\xi+\frac{\xi^{2}}{2m} \leq 2m^{5/6}
\end{align}
for all $v \in A \cup B$.

We will now show, using Lemma \ref{many large matchings}, that there are many ways to build a sequence $(M_1,\ldots ,M_t)$ of edge disjoint perfect matchings in $H$, where each matching contains at least $a$ edges from $E_1$.  To do this, begin by setting $H_0:=H$. Having selected $M_1,\ldots, M_{i-1}$, set $H_i := H \setminus \big ( \cup _{j<i} E(M_{j}) \big )$ and note that $H_i$ is $(r' - i +1)$-regular. Since $r'-i \geq r-t \geq m^{24/25}$ and by \eqref{degrees in E2}, we can apply Lemma \ref{many large matchings} to $H_i$ to find at least $(1-o(1))\big ( \frac{r' - i +1}{e}\big )^m$ perfect matchings of $H_i$ with at least $a$ edges in $E_1$. Multiplying all this estimates gives at least
$$ \prod_{i=1}^t(1-o(1))\Big ( \frac{r' - i +1}{e}\Big )^m={r}^{(1-o(1))tm} = {r}^{(1-o(1))rm}$$ possible choices for $(M_1,\ldots ,M_t)$.

To complete the proof, simply note that each sequence $(M_1,\ldots ,M_t)$ above gives rise to an $(a,t)_{\mathcal M}$-family of $G$, given by ${\bf M} = \{ M_i \cap E_1: i\in [t]\}$. As each $\bf M$ can occur at most $t!$ times in this way, these sequences give rise to
${\cal M} \subset {\cal M}(G, a, t)$ with
$$|{\cal M}| \geq \frac{1}{t!} \times r^{(1-o(1))rm} = r^{(1-o(1))rm}.$$
Lastly, for each such $(a,t)_{\mathcal M}$-family $\bf M$, the minimum degree of $G_{\bf M}$ is at least $t-2m^{5/6}$ by \eqref{degrees in E2}. This completes the proof of the lemma.
\end{proof}

\subsection{Path covers in almost regular oriented graphs}

We are now ready to complete the proof of Lemma \ref{lem: path cover lemma}.

\begin{proof}[Proof of Lemma \ref{lem: path cover lemma}]

Let $b=2\log^4m$ and select a partition $V(H)=V_1\cup \ldots V_b$ uniformly at random, where $|V_i|\in\{\lfloor m/b\rfloor, \lceil m/b\rceil\}$ holds for all $i\in [b]$. For convenience we will assume $|V_i| = m' :=m/b$ for all $i\in [b]$, although this assumption is easily removed. By Chernoff's inequality, with probability $1 - o(1)$ we find that for all
	$v\in V(H)$ and $j\in [b]$ we have
		\begin{equation}
			\label{equation: degree control for the partitioned graphs}
			d^{\pm }_{H}(v,V_{j})
					=
			{d^{\pm}_H(v)}/{b}
			\pm 4\sqrt{m' \log m}
					=
			 d \pm d^{2/3}/2,
		\end{equation}
where $d=r/b$. Fix a choice of partition such that \eqref{equation: degree control for the partitioned graphs} holds.
	
Now consider the complete directed graph on $b$ vertices, denoted by $D_{b}$ (this graph contains both directed edges $(u,v)$ and $(v,u)$ for all pairs of distinct vertices $u, v$). By a result of Tillson \cite{Till}, the complete digraph $D_{b}$ has an edge decomposition into $b$ directed Hamilton paths $Q_1,\ldots ,Q_{b}$. Each such path $Q_i = v_{i_1}\ldots v_{i_{b}}$ naturally corresponds to an oriented subgraph $H_i$ of $H$ consisting of all edges in $B_{ij}:=\overrightarrow{H}[V_{i_j},V_{i_{j+1}}]$ for $j\in [b-1]$. As the paths $\{ Q_i \}_{i\in [b]}$ are edge disjoint, so are the oriented subgraphs $\{ H_i \}_{i\in [b]}$. Note that as $B_{ij}$ only consists of edges oriented from $V_{i_j}$ to $V_{i_{j+1}}$, we can view
$B_{ij}$ as a bipartite graph by ignoring the orientation of its edges.

Our aim now is to show that each oriented graph $H_i$ has many paths covers. Let us fix such a $H_i$ and assume without loss of generality that $H_i$ is given by the path $Q_i = v_1\ldots v_b$, so that $B_{ij} = \overrightarrow{H}[V_j,V_{j+1}]$ for all $j\in [b-1]$. The following observation is key: \vspace{1.5mm}

\begin{obs}\label{obs1}
  Suppose that $M_j$ is a matching of size at least $m'-\ell$ in $B_{ij}$ for all $j\in [b-1]$. Then $\cup M_j$ is a path cover of $H_i$. Moreover, as $\cup M_j$ has at least
 $(m'-\ell)(b-1)$ edges and $H_i$ has $m$ vertices, such path covers are of size at most $m'+b\ell$.
\end{obs}\vspace{1.5mm}

We now exploit this observation using Lemma \ref{lem: partial matching lemma}.
Note that $d-d^{2/3}/2 \geq 2m^{24/25}$. Secondly, by \eqref{equation: degree control for the partitioned graphs} for all $j\in [b-1]$ we have
$$ d-d^{2/3}/2\leq \delta^0(B_{ij})\leq \Delta^0(B_{ij})\leq d+d^{2/3}/2.$$
Therefore, we can apply Lemma \ref{lem: partial matching lemma} to $B_{ij}$, taking
$a' = m' - (m')^{7/8}$ and $t' =d-d^{2/3}/2-(m')^{24/25}$, to get
		\begin{enumerate}[(a)]
			\item \label{matching counts}
				${\mathcal  M}_{ij} \subseteq
				{\mathcal M} (B_{ij}, a',t')$
				with $|{\mathcal M}_{ij}|
				= d^{(1-o(1))dm'}$;
			\item \label{degree count for matchings}
				For all ${\bf M}_{ij} \in {\mathcal M}_{ij}$,
				letting $B:=B_{ij}$, the graph $B_{{\bf M}_{ij}}$
				has minimum degree at least $t'-2(m')^{5/6}$.
		\end{enumerate}

Let us now fix ${\bf M}_{ij}\in \mathcal M_{ij}$ for all $j\in [b-1]$. As each ${\bf M}_{ij}$ consist of $t'$ edge disjoint matchings, by Observation \ref{obs1} we can use $\{{\bf M}_{ij}\}_{j\in [b-1]}$ to construct $t'$ edge disjoint path covers of $H_i$, each of size at most $m' + b(m')^{7/8} \leq n/\log ^4n = a$. Furthermore, it is easy to see that different choices of $\{{\bf M}_{ij}\}_{j\in [b-1]}$ give rise to a different collection of path covers. Combined with \eqref{matching counts}, this gives at least
$$\prod_{j\in[b-1]}\left|\mathcal M_{ij}\right|\geq d^{(1-o(1))(b-1)dm/b}=d^{(1-o(1))dm}$$
distinct $(a,t')_{\mathcal P}$-families of $H_i$.

Now we have partitioned $H$ into $b$ edge-disjoint oriented graphs $H_1,\ldots, H_b$,  each of which consists of at least $d^{(1-o(1))dm}$ distinct $(a,t')_{\mathcal P}$-families. Further, distinct choice of such families from each $H_i$ yield distinct $(a,bt')_{\mathcal P}$-family of $H$. Taking 
$t=bt'\geq  r - 2b(m')^{24/25} \geq r-m^{24/25}\log m$, it follows that there is ${\bf S} \subset {\cal P}(H, a, t)$ with
$$ |{\bf S}| \geq d^{(1-o(1))dmb}=d^{(1-o(1))rm}=r^{(1-o(1))rm}.$$
Here we have used that $b=2\log^4m$, that $d=r/b$ and that $r\geq d/2b$, giving $b^{-rm}=r^{-o(rm)}$. \vspace{1.5mm}

To complete the proof of the lemma, it only remains to prove the following:

\begin{clm}\label{claim minimum degree}
For each ${\bf P} \in {\bf S}$ we have $\delta ^0(H_{\bf P}) \geq r - m/ \log^4m$.
\end{clm}

To see this, simply note that by construction
	$$E ( H_{\bf P} ) = \bigcup _{i,j} E( B_{{\bf M}_{ij}})$$
for some choices of ${\bf M}_{ij} \in {\cal M}_{ij}$ where $i\in [b]$ and $j\in [b-1]$. Given $v \in V_k$ say, the out-edges of $v$ in $H_{\bf P}$ are therefore those out-edges of $v$ in $B_{{\bf M}_{ij}}$, where $i_j = k$. However, $i_j = k$ only occurs when an out-edge of $v_{k}$ appears in $Q_i$, which happens exactly $b-1$ times, since $Q_1,\ldots ,Q_b$ forms a Hamilton path decomposition of $D_{b}$. Combined with \eqref{degree count for matchings}, $t' =d-d^{2/3}/2-(m')^{24/25}$ and $d=r/b$, we find
	\begin{equation*}
		d^{+}_{H_{\bf P}}(v) \geq (b-1)(t' - 2(m')^{5/6}) \geq bt' - t' - 2b(m')^{5/6} \geq r-t'-4b(m')^{24/25}\geq r-2m'= r - m/\log ^4m.
	\end{equation*}
As an identical argument lower bounds the $d^{-}_{H_{\bf P}}(v)$, this completes the proof of the claim, and hence the proof of the lemma.
\end{proof}

\section{Partitions of oriented graphs}

	In this final section before the proof of Theorem \ref{main} and Theorem \ref{main2}
	we prove a technical lemma which will allow us to decompose oriented graphs as given
	in Theorem \ref{main} into smaller subgraphs, each of which satisfy the hypothesis of
	Lemma \ref{lem: completion of many hamilton cycles} and Lemma \ref{lem: path cover lemma}.

\begin{lem}\label{partition lemma}
	Let $\beta\geq \alpha > \varepsilon >0$, let $K, d, n \in {\mathbb N}$, with $n$ sufficiently
	large, $d = \alpha n$ and $K = \log n $. Suppose that $G$ is an oriented graph with
	$\delta^0(G)\geq \beta n$ and that $D$ is a $d$-factor of $G$. Then there
	are $K^3$ edge-disjoint spanning subgraphs
	$H_1,\ldots ,H_{K^3}$ of $G$ with the following properties:
		\begin{enumerate}
			\item For each $H_i$ there is a partition $V(G) = U_i \cup W_i$ with
			$|W_i|  =  n/K^2 \pm 1$;
			\item Letting $D_i = H_i[U_i]$, for some
			$r \geq (1-2\varepsilon )d/K^3$ we have
			$$r -r^{3/5} \leq  \delta^0(D_i)
			\leq \Delta^0(D_i)\leq r +r^{3/5};$$
			\item Letting $E_i = H_i[U_i, W_i]$ we have
			$d^{\pm} _{E_i} (u,W_i) \geq \varepsilon |W_{i}|/4K$ for all $u \in U_i$;
			\item Letting $F_i = H_i[W_i]$ we have $\delta ^{0}(F_i)
			\geq (\beta - \varepsilon )|W_i| $.
		\end{enumerate}
\end{lem}

\begin{proof}
	To begin, select $K$ partitions of $V(G)$ uniformly and independently
	at random where, for each $k\in [K]$, we partition $V(G)$ into $K^2$ sets,
	$V(G) = \bigcup _{\ell \in [K^2]} S_{k,\ell}$ with $|S_{k,\ell }| \in \left\{
	\lfloor n/K^2 \rfloor,\lceil n/K^2 \rceil\right\}$. Note that for each $k\in [K]$ and $v\in V(G)$ there exists a unique $\ell:=\ell(k,v)\in [K^2]$ for which $v\in S_{k,\ell}$. In particular, every $v\in V(G)$ belongs to \emph{exactly} $K$ sets $S_{k,\ell}$.

Second, observe that by Chernoff's inequality for a hypergeometrical distribution (see Remark \ref{rem:hyper}), letting $s=\lfloor n/K^2\rfloor$, with
	probability $1-nK^3e^{-\omega(\log n)}=1 - o(1)$ we have
		\begin{equation}
			\label{equation: degree to S_{k,j}}
			d^{\pm }_{D}(v,S_{k,\ell})
				=
			\alpha|S_{k,\ell}| \pm 4 \sqrt {s\log n}
			\quad \text{ and }\quad
            d^{\pm}_G(v,S_{k,\ell})= d_G^{\pm}(v)|S_{k,\ell}|/n\pm 4\sqrt{s\log n}
		\end{equation}
	for all $v \in V(G)$, $k\in [K]$ and $\ell\in[K^2]$. In particular, as $|S_{k,\ell }| = s \pm 1 > n/2K^2 \gg \log n$, for all $k$ and $\ell$ we have
\begin{equation}\label{eq:bound on delta}\delta^{0}(G[S_{k,\ell}])\geq \beta |S_{k,\ell}| - 4 \sqrt {s\log n} \geq
(\beta-\varepsilon/2)|S_{k,\ell}|.\end{equation}

For each $v\in V(G)$ and $k\in [K]$, let $X^{+}(v,k)$
denote the random variable which counts the number of $w\in N^+_G(v)$ such that $w\in S_{k,\ell(k,v)} \cap  S_{k',\ell(k',v)}$ for some $k'\neq k$. Define $X^-(v,k)$ similarly.

Note that for $\sigma\in\{+,-\}$ we have
$${\mathbb E} [X^{\sigma}(v,k)]
	\leq K \Big ( \frac{n}{K^4} \Big ) =\frac{n}{K^3}=o(s).$$ By Chernoff's inequality,
	with probability $1 - Kne^{-\Theta(n/K^3)}=1-o(1)$, for all $k\in [K]$ and $v \in V(G)$ we have
		\begin{equation}
			\label{equation: number of repeated edges}
		 X^{\sigma}(v,k)
		 	\leq \frac {2n}{K^3} = o(s).
		\end{equation}

Lastly, for $\sigma\in \{+,-\}$ and $v\in V(D)$ we define the random variable $Y^{\sigma}(v)$ to be the set of all vertices $u\in N^{\sigma}_D(v)$ with $u\in S_{k,\ell(k,v)}$ for some $k$. For all $\sigma\in\{+,-\}$ and $v \in V(D)$ we have $$b: = {\mathbb E} \left[ Y^{\sigma}(v)\right] \leq Ks.$$
Note that, since all the vertices of $D$ have the same in/outdegrees, the value of ${\mathbb E} \left[ Y^{\sigma}(v)\right]$ is indeed independent of $v$.
	By Chernoff's inequality, with probability
	$1 - 2nK e^{-(2\sqrt {Ks \log n})^2/3Ks}=1-o(1)$, for all
	$\sigma\in\{+,-\}$ and $v \in V(D)$ we have
		\begin{equation}
			\label{equation: number of edges within an S set}
		 Y^{\sigma}(v)
		 	= b \pm 2 \sqrt
		 {Ks \log n}.
		\end{equation}
	Thus, with positive probability a collection of partitions satisfy
	\eqref{equation: degree to S_{k,j}},
	\eqref{equation: number of repeated edges}
	and \eqref{equation: number of edges within an S set}.
	Fix such a collection.
	
	We relabel $\{S_{k,\ell}\mid k\in[K] \text{ and } \ell\in[K^2]\}$ as
	$\{W_1,\ldots, W_{K^3}\}$ (arbitrarily). Also set $F_i=G[W_i]\setminus R_i$,
	where $R_i$ is the set of all edges which appear in more than one $W_i$.
	From \eqref{eq:bound on delta} and \eqref{equation: number of repeated edges},
	for each $i\in [K^3]$ we obtain
	$$\delta^0(F_i)\geq (\beta-\varepsilon)|W_i|.$$

	Next, let $D'=D\setminus \left(\bigcup_i E(G[W_i])\right)$. As $D$ is $d$-regular,
	by \eqref{equation: number of edges within an S set},
	we have that for all $\sigma \in \{+,-\}$ and $v\in V(D)$
		\begin{align*}
			d^{\sigma }_{D'}(v)
				=
			d^{\sigma }_{D}(v) - Y^{\sigma }(v)
				=
			d - b \pm 2 \sqrt {Ks\log n}.
		\end{align*}
	
	To complete the proof we partition the edges of $D'$ into further oriented subgraphs
	$$\{D_{i}\}_{i\in [K^3]} \text{ and }\{ E_{i} \}_{i\in [K^3]}.$$

Each $D_{i}$ will be an oriented subgraph with $V(D_{i})=V(D) \setminus W_i:=U_i$, and each $E_{i}$ will consist of some directed edges between $U_{i}$ and $W_{i}$. To obtain these graphs we will partition the edges at random as follows:
	Suppose that $e = uv\in E(D')$, and let $I_u=\{i\in [K^3] \mid u\in W_i\}$. Similarly, define $I_v$.
	By construction, $|I_u|=|I_v|=K$ and $I_u\cap I_v=\emptyset$.
Now, we randomly and independently assign each $e \in E(D')$ to a subgraph according to the following distribution:
		\begin{itemize}
			\item for $i\notin I_u\cup I_v$, we assign $e$ to $D_{i}$ with
			probability $\frac{1-\varepsilon }{K^3-2K}$;
			\item for $i\in I_u\cup I_v$, we assign $e$ to $E_{i}$ with
			probability $\frac{\varepsilon }{2K}$.
		\end{itemize}
Note that the probability for $e$ to being assigned to some subgraph is $1$.

By Chernoff's inequality, with probability at least $1 - nK^3e^{-\Theta(\sqrt{n\log n})^2/n} - nK^3e^{- \Theta(\frac { s }{K})}=1-o(1)$ the resulting oriented graphs satisfy
		\begin{enumerate}[(a)]
			\item $r - r^{3/5} \leq r-4\sqrt{n\log n}\leq \delta^0(D_i)\leq \Delta^0(D_i)\leq r+ 4\sqrt{n\log n} \leq r + r^{3/5}$, where $r :=
			\frac {(1-\varepsilon )(d-b)}{K^3-2K} \geq \frac{(1-2\varepsilon )d}{K^3}$;
			
			\item $d^{\pm }_{E_{i}}(v,W_{i}) \geq \varepsilon|W_{i}|/4K$
			for all $v\in U_{i}$.
		\end{enumerate}
	Finally, taking $H_i=D_i\cup E_i\cup F_i$ for each $i\in [K^3]$, it is easy to check that these graphs satisfy the requirements.\end{proof}

\section{Proof of Theorem \ref{main}}

We are now ready to complete the proof of Theorem \ref{main}.

\begin{proof}[Proof of Theorem \ref{main}] Let $G$ be an oriented graph as in the assumptions of the theorem. Let $d := \reg(G) = \alpha n$ and let $D\subseteq G$ be a $d$-factor of $G$. From Theorem \ref{thm:exact semidegree condition}, we find that $G$ contains $(c - 3/8)n$ edge disjoint Hamilton cycles, and so $\alpha \geq c-3/8 >0$.

First, we apply Lemma \ref{partition lemma} to $G$ and $D$, with $\beta = c$, $\alpha $ and $\varepsilon /4$ in place of $\varepsilon$. Setting $K= \log n $, this gives edge-disjoint subgraphs $H_1,\ldots, H_{K^3}$ of $G$ with the following properties:
\begin{enumerate}
			\item For each $H_i$ there is a partition $V(G) = U_i \cup W_i$ with
			$|W_i|  =  n/K^2 \pm 1$;
			\item Letting $D_i = H_i[U_i]$, for some $r \geq (1-\varepsilon /2)d/K^3$,
			we have
			$$r - r^{3/5} \leq  \delta^0(D_i)\leq \Delta^0(D_i) \leq r + r^{3/5};$$
			\item Letting $E_i = H_i[U_i, W_i]$ we have
			$d^{\pm} _{E_i} (u,W_i) \geq \varepsilon |W_{i}|/4K$ for all $u \in U_i$;
			\item Letting $F_i = H_i[W_i]$ we have $\delta ^{0}(F_i)
			\geq (\beta - \varepsilon )|W_i| $;
		\end{enumerate}
	
Secondly, by property $2.$ above we can apply Lemma \ref{lem: path cover lemma} to each oriented graph $D_i$. This gives ${\bf P}_i\in \mathcal P(D_i,n/\log^4n, r-n/\log^{4}n)$ which satisfies
	\begin{equation}
		\label{control of path graphs}
		\delta^0(D_{{\bf P}_i}) \geq r-n/\log^{4}n.
	\end{equation}

Lastly, apply Lemma \ref{lem: completion of many hamilton cycles} to ${\bf P}_i$ for each $i$. Taking $t=r-n/\log^{4}n$ and $a = b=n/\log^4n$ and $s = |W_i| = n/K^2 \pm 1$, it is easy to check that the conditions of Lemma \ref{lem: completion of many hamilton cycles} hold using \eqref{control of path graphs} and properties 3. and 4. above. This gives a collection $\mathcal C_i:=\{C_{i1},\ldots,C_{it}\}$ of edge-disjoint Hamilton cycles  in $H_i$.

To complete the proof, set $\mathcal C:=\bigcup_i \mathcal C_i$. Since the $H_i$ are edge-disjoint, together with property $3.$, we find that $\mathcal C$ consists of
$$K^3t\geq (1-\varepsilon /2)K^3r \geq (1-\varepsilon)d$$
edge-disjoint Hamilton cycles of $G$. This completes the proof. \end{proof}

\section{Proof of Theorem \ref{main2}}

Before proving Theorem \ref{main2} let us introduce a final convenient definition.

\begin{dfn}
	Given an oriented graph $H$, a collection of $t$ edge-disjoint Hamilton cycles $\{C_1,\ldots,C_t\}$ of $G$ is called an $(H,t)_{\mathcal C}$-family. Let $\mathcal C(H,t)$ denote the set of all $(H,t)_{\mathcal C}$-families of $H$.
\end{dfn}

We are now ready for the proof of Theorem \ref{main2}.

\begin{proof}[Proof of Theorem \ref{main2}] Let $c>3/8$ be fixed and $d = cn$. We would like to show that given any $\varepsilon >0$ and a large enough $n$, every $d$-regular oriented graph $G$ on $n$ vertices satisfies
	\begin{equation*}	
		\Big |\mathcal C(G,d)\Big |
			\geq
		n^{(1-\varepsilon )dn}.
	\end{equation*}
Let $K = \log n$ and $\alpha = \varepsilon /4$. Our proof proceeds in five steps.\vspace{1.5mm}

\noindent \textbf{Step 1.} Removing a $\delta$-absorbing subgraph from $G$.\vspace{1mm}
	
By Theorem \ref{thm: Kuhn-Osthus Absorbing}, there exists $\delta>0$ such that $G$ contains a $\delta$-absorber subgraph $A$, where $A$ is $a$-regular, with $a\leq \alpha n$. Fix such a choice of $A$ and let $G_0:=G\setminus A$.
\vspace{3mm}

{\bf Step 2.} Partitioning $G_0$.\vspace{1.5mm}
	
Note that $G_0$ is $d':=cn-a$ regular with $\beta:=d'/n>3/8$. Therefore, taking $D=G$ and $\varepsilon _0 = \varepsilon /10$, applying Lemma \ref{partition lemma}, one can find $K^3$ edge-disjoint spanning subgraphs $H_1,\ldots,H_{K^3}$ of $G_0$ satisfying:
\begin{enumerate}
			\item For each $H_i$ there is a partition $V(G) = U_i \cup W_i$ with
			$|W_i| =n/K^2 \pm 1$;
			\item Letting $D_i = H_i[U_i]$, with $r \geq (1-2\varepsilon_0 )d'/K^3$, we have $$r -r^{3/5}\leq  \delta^0(D_i)\leq \Delta^0(D_i)\leq
			r +r^{3/5};$$
			\item Letting $E_i = H_i[U_i, W_i]$ we have
			$d^{\pm} _{E_i} (u,W_i) \geq \varepsilon_0 |W_{i}|/4K$ for all $u \in U_i$;
			\item Letting $F_i = H_i[W_i]$ we have $\delta ^{0}(F_i)
			\geq (\beta - \varepsilon _0)|W_i| $.
		\end{enumerate}

\noindent \textbf{Step 3.}  Showing that for some $t=r-o(r)$ and for every $i\in [K^3]$ the set ${\mathcal C}(H_i,t) $ is large. \vspace{1.5mm}

To this end, let us first apply Lemma \ref{lem: path cover lemma} to each of the $D_i$s (note that by Property $3$ above, the assumptions are fulfilled, and that $|U_i|=m=(1-o(1))n$). It thus follows that for every $i$ we have a collection
$$\mathcal P_i\subseteq  \mathcal P(D_i,n/\log^4n,r-n/\log^{4}n) $$
which satisfies
$$\left| \mathcal P_i \right|\geq r^{(1-o(1))rn},$$
such that $\delta^0(D_{{\bf P}_i})\geq r-n/\log^{4}n$ for all ${\bf P}_i\in \mathcal P_i$.

Therefore, by Properties $4$, $5$ and the lower bound on $\delta ^0(D_{{\bf P}_i})$, the hypothesis of Lemma \ref{lem: completion of many hamilton cycles} apply to $H_i$ and ${\bf P}_i$, taking $a = b=n/\log^{4}n$, $t = r - n/\log ^4n$ and $s = |W_i| = n/K^2 \pm 1$. This lemma allows us to turn ${\bf P}_i$ into a collection of $t=r-n/\log^4 n$ edge-disjoint Hamilton cycles. Noting that we fix the $W_i$ sets throughout the proof, we can trivially recover the path cover used to build each of the cycles. Therefore, for all $i\in [K^3]$ we have
	\begin{equation}
		\label{equation: HC count}
		|\mathcal C(H_{i},t)|\geq \left|\mathcal P_i\right|
			\geq
		r^{(1-o(1))rn}.
	\end{equation}

\noindent \textbf{Step 4.} Showing that $G_0$ has $n^{(1-\varepsilon)dn}$ `almost Hamilton decompositions'.\vspace{1.5mm}
	
To see this, note that if we pick ${\bf C}_i\in \mathcal C(H_i,t)$ for all $i$, then ${\bf C}=\bigcup_i {\bf C}_i\in \mathcal C(G_0,K^3t)$.
Therefore, by \eqref{equation: HC count}, for $t'=K^3t$ we conclude that
\begin{align}\label{eq: count}\left|\mathcal C(G_0,t')\right|&\geq r^{(1-o(1))rnK^3}\geq d^{(1-\varepsilon/5)d'n} \geq n^{(1-\varepsilon/4)(1-\alpha)dn}\geq n^{(1-\varepsilon /2)dn}.
\end{align}
	
	\noindent \textbf{Step 5.} Completing every ${\bf C}\in {\cal C}(G_0,t')$ to a
	Hamilton decomposition of $G$.\vspace{1.5mm}
	
	Let ${\bf C}\in {\cal C}(G_0,t')$ and note that $G'=G_0\setminus {\bf C}$ is a
	$b$-regular oriented graph with $b=o(n)$. Since $A:=G\setminus G_0$ is a
	$\delta$-absorber, and $b<\delta n$, it follows from Theorem
	\ref{thm: Kuhn-Osthus Absorbing} that $A \cup G'$ has a Hamilton decomposition
	${\bf C}'$. But then ${\bf C} \cup {\bf C}'$ is a Hamilton decomposition of
	$G$. Lastly, note that although different choices of ${\bf C} \in {\cal C}(G_0,t')$
	may give rise to the same Hamilton decomposition in this way, it is easy to see
	that each such decomposition occurs at most $\binom {d}{t'} \leq 2^n$ times.
	By \eqref{eq: count}, this gives
		\begin{equation*}
			|{\cal C}(G,d)| \geq |{\cal C}(G_0,t')| /2^n \geq n^{(1-\varepsilon )dn}.
		\end{equation*}	
	This completes the proof.
\end{proof}
	
\section{Concluding remarks}

In this paper we have given bounds on the number of Hamilton decompositions of dense regular oriented graphs. Theorem \ref{main} shows that if $G$ is an $r$-regular $n$-vertex oriented graph, with $r = cn$ for some fixed $c> 3/8$, then it has $r^{(1+o(1))rn}$ Hamilton decompositions. As indicated in the Introduction this bound is tight for every such graph, up to the $o(1)$-term in the exponent. 

We believe that such oriented graphs should in fact have $\big ((1+o(1))\frac{r}{e^2} \big )^{rn}$ Hamilton decompositions. This would agree with the more precise upper bound obtained from the Minc conjecture in the Introduction. To prove this seems to require a version of Theorem \ref{thm: Kuhn-Osthus Absorbing} which can be applied to oriented graphs with sublinear density. In this respect, it would be very interesting to obtain an alternative proof of Kelly's conjecture that does not make use of regularity, as it seems likely to lead to such a theorem.

\end{document}